\documentclass[12pt, twoside]{article}

\usepackage{amsthm,amsmath, amssymb}
\usepackage{color,tikz}
\usepackage{graphicx}
\usepackage{enumerate}
\usepackage{titlesec}
\usepackage[a4paper,bindingoffset=0.5cm,left=1.8cm,right=1.8cm,top=2cm,bottom=2cm,footskip=.8cm]{geometry}
\usepackage{fancyhdr}
\usepackage{multicol}

\makeatletter
\renewcommand*{\@fnsymbol}[1]{\ensuremath{\ifcase#1\or *\or ** \or \dagger\or \ddagger\or
   \mathsection\or \mathparagraph\or \|\or **\or \dagger\dagger
   \or \ddagger\ddagger \else\@ctrerr\fi}}
\makeatother

\numberwithin{equation}{section}
\theoremstyle{plain}
\newtheorem{thm}{Theorem}[section]
\newtheorem{lemma}{Lemma}[section]
\newtheorem{corollary}{Corollary}[section]

\newcommand{\e}{{\mathrm e}}

\newcommand{\Q}{{\mathbb Q}}
\newcommand{\Z}{{\mathbb Z}}
\newcommand{\D}{{\mathrm d}}
\newcommand{\M}{{\boldsymbol M}}
\newcommand{\dd}{{\mathrm d}}

\newcommand{\NHPP}{{\sc NHPP}}

\newcommand{\E}{\mathbb{E}}
\newcommand{\PR}{\mathbb P}
\newcommand{\PGF}{{\sc pgf}}

\newcommand{\MCLT}{{\sc mclt}}
\newcommand{\FCLT}{{\sc fclt}}
\newcommand{\Var}{\mathbb{V}\mathrm{ar}}
\newcommand{\Cov}{{\mathbb C}\mathrm{ov}}

\titleformat{\section}{\normalfont\fontsize{12}{12}\bfseries}{\thesection}{1em}{}
\titleformat{\subsection}{\normalfont\fontsize{12}{12}\bfseries}{\thesubsection}{1em}{}

\date{}
\title{{\bf Scaling limits for infinite-server systems\\ in a random environment}}
\author{By Mariska Heemskerk\thanks{University of Amsterdam}, Johan van Leeuwaarden\thanks{Eindhoven University of Technology} \\and Michel Mandjes\footnotemark[1]}
\AtEndDocument{\bigskip{\footnotesize
\begin{multicols}{2}
\noindent 
University of Amsterdam \\ Korteweg-de Vries Institute for Mathematics \\ Science Park 107 \\1098 XG  Amsterdam\\ The Netherlands\\
E-mail: j.m.a.heemskerk@uva.nl
\\
\hspace*{32pt} m.r.h.mandjes@uva.nl
\columnbreak

\noindent
Eindhoven University of Technology\\ Department of Mathematics and Computer Science\\ PO Box 513\\ 5600 MB Eindhoven\\ The Netherlands\\
E-mail: j.s.h.v.leeuwaarden@tue.nl 
\end{multicols}
}}

\begin{document}

\maketitle

\vspace*{-35pt}
\begin{abstract}\noindent \textbf{Abstract.}
This paper studies the effect of an overdispersed arrival process on the performance of an infinite-server system.
In our setup, a random environment is modeled by drawing an arrival rate $\Lambda$  from a given  distribution every $\Delta$ time units, yielding an i.i.d.\ sequence of arrival rates $\Lambda_1,\Lambda_2, \ldots$.
Applying a martingale central limit theorem, we obtain a functional central limit theorem for the scaled queue length process.
We proceed to large deviations and derive the logarithmic asymptotics of the queue length's tail probabilities. 
As it turns out, in a rapidly changing environment (i.e., $\Delta$ is small relative to $\Lambda$) the overdispersion of the arrival process hardly affects system behavior, whereas in a slowly changing random environment it is fundamentally different; this general finding applies to both the central limit and the large deviations regime.
We extend our results to the setting where each arrival creates a job in multiple infinite-server queues.
\end{abstract}


\section{Introduction}

Empirical studies show that the number of arrivals in customer contact centers, hospital emergency departments and cloud computing systems typically varies strongly over time \cite{KW14, WGK07}.
This motivates modeling such arrival processes by a \textit{non-homogeneous Poisson process} (\NHPP) with time-dependent arrival rate $\lambda(t)$, see e.g.\ \cite{KW13}.
At the same time, various studies show that in a broad variety of real-life systems the intensity of the fluctuations in the arrival rate is so severe that the Poisson assumption ceases to hold \cite{BRZ10,KW14}.
The observed level of \textit{overdispersion} urges the need to develop stochastic models that can capture such persistent fluctuations. 

Starting from the classical Poisson process, it is common practice to increase dispersion by using a \textit{mixed} Poisson process \cite{BRZ10,M09}, to that end replacing a deterministic parameter $\lambda$ by a random parameter $\Lambda$.
This leads to the idea of modeling overdispersed arrival processes by a mixed version of NHPPs, so-called \textit{Cox processes} \cite{Cox55}, where the time-dependent rate $\lambda(t)$ of the classical NHPP is replaced by a stochastic process $\Lambda(t)$.
For instance, one could use Markov-modulated Poisson processes in which the arrival rate $\Lambda(t)=\lambda_{J(t)}$ is a function of a continuous-time Markov chain $J(\cdot)$ on a finite state space $S$ and non-negative rates $\lambda_i$ for $i \in S$ (see e.g.\ \cite{ABM14, BTK14}). 
To also include, say, diurnal patterns, one could 
work with the arrival rate $\Lambda(t)+\bar\lambda(t)$ for some function $\bar\lambda(t)$.  
Although the Markov-modulated Poisson process is versatile and has various attractive properties, it has considerable drawbacks as well.
First, while a substantial body of results for single-server queues with Markov modulation has been established, considerably less is known about their many-server and infinite-server counterparts; see e.g.\ an account of this issue for the infinite-server system in \cite{BTM15}. 
Second, due to the fact that the process $J(\cdot)$ is not observed, estimating the parameters of a Markov-modulated Poisson process from data is a non-trivial task \cite{RYD96}.

The main objective of this paper is to develop an arrival process simpler than a Markov-modulated Poisson process -- arguably the simplest in terms of analysis -- that fits the overdispersed and time-dependent setting, and to assess the impact of these characteristics on a corresponding system's performance. 
The model we propose is a {\it mixed Poisson} arrival process in a {\it random environment}.
It is defined as follows.
Let $\Lambda$ a non-negative random variable with finite first two moments and density $f_{\Lambda}(\cdot)$.
Introduce a \textit{sampling frequency} $\frac{1}{\Delta}$; then the arrival rate at time $t$ is given by $\Lambda_j$ when $t \in [j\Delta,(j+1)\Delta)$, where the $\Lambda_j$ are independent and distributed as a non-negative random variable $\Lambda$, for $j \in \Z$.
In other words, this arrival process is a special case of a stationary Cox process where the arrival rate at time $t$ is given by 
\begin{equation}\label{deflab}
\Lambda(t)=\sum_j \Lambda_j 1_{[j\Delta,(j+1)\Delta)}(t).
\end{equation}
To add nonstationarity in the arrivals, one could include a deterministic component $\bar\lambda(t)$ without intrinsically complicating the analysis; for ease of presentation we omit the extra component here.
The resulting process can be viewed as an extension of the classical mixed Poisson setting, which is enriched by (independently) resampling the arrival rate after every time slot of length $\Delta>0$. 
The intuition is that the arrival rate changes every $\Delta$ time units, so that the number over a large time slot fluctuates more severely than standard Poisson data would, as can be made explicit via an elementary computation.
Let the number of arrivals up to time $t$ be given by $N_t \sim {\rm Pois}(\int_0^t \Lambda_s\, \dd s)$ and let $t$ be some multiple of $\Delta$ (for simplicity).
Then $\E N_t = t \E \Lambda$, whereas
\begin{align*}
\Var(N_t)=\sum_{j=1}^{t/\Delta} \Var(N_{\Delta})&=t \Delta^{-1}\big(\E [ \Var(N_{\Delta} | \Lambda)] + \Var(\E[N_{\Delta}|\Lambda])\big)\\
&=t\big(\E \Lambda + \Delta \Var(\Lambda)\big).
\end{align*}
Conclude that, as desired, the variance-to-mean ratio is strictly larger than $1$ for non-deterministic $\Lambda$, i.e.,
\[
\frac{\Var(N_t)}{\E N_t}=1 + \Delta \frac{\Var(\Lambda)}{\E \Lambda}.
\] 
Observe that the level of overdispersion is determined by the interval length $\Delta$ and the level of overdispersion in $\Lambda$ (through its variance-to-mean ratio).

Given this model for the arrival process, various queueing models can be studied; in this paper we focus on single-class infinite-server systems with exponential service times. 
The proposed arrival process being overdispersed, the main objective of this paper is to reveal, in a compact manner, the impact of overdispersion on system performance.
Infinite-server systems are a natural choice when the system at hand is designed to (almost) immediately serve all customers \cite{Whitt99}, but it may also serve as a tractable proxy for the more complicated multi-server systems, which is for instance exploited in the modified offered-load (MOL) and pointwise stationary approximation (PSA) methods for staffing large-scale service systems in a time-varying setting \cite{WK12,WGK07}. \\


\noindent \textbf{Contributions.} Infinite-server systems with overdispersed arrivals are, as described above, very tractable.
As shown in Section \ref{sec:OVDIS}, it is fairly straightforward to compute the probability generating function (\PGF{}) of the stationary and time-dependent queue length processes in terms of transforms. 
This is due to the fact that customers are served immediately upon arrival, independently of each other; as a result, when analyzing the queue length at a given point in time, we can separately consider the individual (independent!) contributions that correspond to each of the preceding intervals of length $\Delta$.

The queue length distribution can be characterized in terms of its \PGF{}, which effectively means that evaluation of the accompanying performance measures requires numerical inversion.
However, by imposing a scaling on both the time and scale parameters, $\Delta$ and $\Lambda$, we succeed in identifying an asymptotic regime in which the distribution {\it can} be explicitly given.
We inflate the arrival rate and sampling frequency in the following way:
\begin{align}\label{scale}
\begin{aligned}
\Lambda \mapsto N\Lambda \qquad
\Delta^{-1} \mapsto N^{\alpha}\Delta^{-1},
\end{aligned}
\end{align}
where we let $N\to\infty$.
Importantly, $\Lambda$ and $\Delta^{-1}$ do not necessarily grow at the same rate under scaling (\ref{scale}). 
The value of $\alpha$ determines the asymptotic behavior of the resulting scaled system, giving rise to a {\it trichotomy}.
For $\alpha > 1$, in which case the arrival rate is resampled relatively frequently, we find that the system behaves as a standard infinite-server queue (no overdispersion), whereas for $\alpha < 1$ the overdispersion remains present in the asymptotic regime. 
The case $\alpha = 1$ essentially reflects a superposition of the two distinct types of behavior.

For preparatory purposes, we show in Section \ref{sec:OVDIS} that the centered and normalized stationary queue length is asymptotically normal under the scaling in (\ref{scale}). 
Next, in Section 3 we consider a multidimensional setting with correlated arrivals: an arrival triggers jobs in multiple queues. 
Hence, we work with a coupled system in which $d$ parallel queues are fed by a single arrival process; cf.\ \cite{LP15a,LP15b}.
With ${\boldsymbol U}^{(N)}(\cdot)$ denoting the vector of centered and normalized queue length processes, the asymptotic normality now translates to the corresponding limiting process ${\boldsymbol U}(\cdot)$ being Gaussian: ${\boldsymbol U}(\cdot)$ is a $d$-dimensional process of the Ornstein-Uhlenbeck type with parameters that depend on the scaling regime. 
Following the approach in \cite{ABM14}, we show this by applying a lemma due to Kurtz and a martingale central limit theorem (\MCLT) to a suitable stochastic integral equation.

Subsequently, in Section \ref{sec:LD} we carry out a large deviations analysis to obtain the logarithmic tail asymptotics corresponding to the queue length distribution. 
The crucial observation in this analysis is that rare events can essentially be realized in two ways: (i)~the random arrival rate attains an exceptionally high value, (ii)~the Poisson process generates an unusually large number of arrivals given the (not so rare) value of the random parameter.
Again, the value of $\alpha$ determines what type of tail behavior dominates: for $\alpha<1$ this is effect~(i), for $\alpha>1$ effect~(ii), and for $\alpha=1$ a combination of effects (i) and~(ii).
These findings complement similar results that have been established for an infinite-server system with Markov-modulated input, where it is noted that the slow regime ($\alpha\in(0,1)$) was not covered in that setting \cite{BTK14, TM14}. 
We conclude Section \ref{sec:LD} by pointing out how the large deviations results can be extended to the multidimensional setting.

\section{Overdispersion in an infinite-server context}\label{sec:OVDIS}
In this section we present a stationary and transient analysis of the single-class Markovian infinite-server system in a random environment just introduced. 
A crucial role is played by $\Lambda(t)$, the arrival rate at time $t$ given in (\ref{deflab}).
Remember that we assumed that the arrival rates are i.i.d.\ and distributed as a random variable $\Lambda\geqslant 0$ with finite first two moments and density $f_{\Lambda}(\cdot)$.
The corresponding service times are assumed i.i.d.\ (and in addition independent of the arrival process) exponentially distributed random variables with mean $1/\mu$.

First, in Section \ref{ss2.1}, we analyze the stationary system behavior, in terms of its \PGF{} and the corresponding moments, which
we then extend to the associated transient behavior. 
We then study the stationary behavior in a central limit regime under parameter scaling (\ref{scale}) in Section \ref{ss2.2}.
This exposition serves as an illustration for the reader, and is intended to create intuition as for why the scaled stationary queue length is asymptotically normal and why the three different limiting regimes appear; in addition, in Section \ref{sec:LD} we need a result that is proven along the same lines.
We remark that in Section \ref{sec3} the normality is generalized in several directions: we establish a {\it functional} central limit theorem (\FCLT) for the (scaled) transient process $M^{(N)}(\cdot)$ corresponding to the $d$-dimensional parallel system as defined in the introduction.

\subsection{Pre-limit results} \label{ss2.1}
This subsection presents `pre-limit results'; later we study their counterparts in the limiting regime after imposing a parameter scaling.\\

\noindent \textbf{Transform of stationary queue length.}
Let $M$ be the random variable associated with the stationary number of jobs (also sometimes referred to as `customers') in the system.
Exploiting `thinning' properties, we can identify the \PGF{} $\phi(z):= \E z^M$ of $M$. 

In the sequel we write $p_t:={\rm e}^{-\mu t}$ for the probability that a job present at $kt$ is still present at $(k+1)t$ and $q_t:=(1-{\rm e}^{-\mu t})/(\mu t)$ for the probability that a job arriving at a uniform epoch in $[kt, (k+1)t)$ is still present at $(k+1)t$.  
Denote $\bar p_t:=1-p_t$. 

Note that $M$ can be written as  the sum of $M_0, M_1, M_2,\ldots$, where $M_k$ represents the number of jobs that arrived in $[-(k+1)\Delta, -k\Delta)$ and are still present at time $0$.
Furthermore, observe that these `thinned' random variables $M_k$ are independent. 
A job that arbitrarily arrived in $[-(k+1)\Delta, -k\Delta)$ (i.e., having arrived at a uniform epoch in this interval) is still in the system at time $0$ with probability
\begin{equation*}
\int_0^\Delta\frac{1}{\Delta} {\rm e}^{-\mu(k\Delta+s)}{\rm d}s 
=q_{\Delta} p_\Delta^{k}.
\end{equation*} 
As a consequence, with $r_t:=t q_t$,
\begin{align}\nonumber
\phi_k(z):= {\mathbb E}z^{M_k} &= \sum_{\ell=0}^\infty \int_0^\infty f_\Lambda(\lambda){\rm e}^{-\lambda \Delta}\frac{(\lambda \Delta)^\ell}{\ell!} 
\sum_{m=0}^\ell z^m {\ell \choose m} \big(q_\Delta p_\Delta^{k}\big)^m \big(1-q_\Delta p_\Delta^{k} \big)^{\ell -m} 
{\rm d}\lambda\\ \nonumber
&=\int_0^\infty  \exp\big(-\lambda r_\Delta p_\Delta^{k}(1-z)\big)f_\Lambda(\lambda)\,{\rm d}\lambda\\ 
&=\E\exp\big(-\Lambda_k r_{\Delta}p_{\Delta}^k(1-z)\big). \label{pgfk}
\end{align}
Observe that $\phi_k(z)$ is a \PGF{} of `mixed Poisson' type: conditional on $\Lambda_k=\lambda$ the \PGF{} corresponds with that of a Poisson random variable with mean $\lambda  r_\Delta p_\Delta^k$. 
We conclude that $M_k$ is distributed as a mixed Poisson random variable with random parameter
 \[
\kappa_{k}(\Lambda_{k}):= \Lambda_{k} r_\Delta p_\Delta^k , 
\]
with $\Lambda_k$ the value of the arrival rate in the interval $[-(k+1)\Delta, -k\Delta)$ (note that, in fact, we should have written  $\Lambda_{-(k+1)}$ rather than $\Lambda_k$, but
due to the i.i.d.\ assumption the processes $\{\Lambda(s)\}_{s\geqslant0}$ and $\{\Lambda(-s)\}_{s \geqslant0}$ have the same finite-dimensional distributions).
Therefore, $M$ is mixed Poisson as well and its random parameter is given by
\begin{align}\label{pars}
\sum_{k=0}^\infty \kappa_k(\Lambda_k)=\int_0^{\infty} \Lambda(s) \e^{-\mu s}\, \dd s =:\kappa(\Lambda).
\end{align}
(Note that $\kappa(\cdot)$ is defined as a functional; $\kappa(\Lambda)$ should be interpreted as $\kappa(\Lambda(\cdot))$.)

There is an alternative way to obtain this result.
Indeed, since we observe the system in stationarity,
\begin{align}\label{fp}
\phi(z)
&=\phi(zp_\Delta+\bar p_\Delta)g_{\Lambda,\Delta}(z),
\end{align}
where $g_{\Lambda,t}(z)$ is defined by
\[g_{\Lambda,t}(z):=
\int_0^\infty  \exp\big(-\lambda r_t(1-z)\big)f_\Lambda(\lambda)\,{\rm d}\lambda=\E\exp\big(-\Lambda r_t(1-z)\big).
\]
Applying an iteration argument to (\ref{fp}) yields 
\begin{align}\label{pgf}
\phi(z) = \prod_{k=0}^\infty g_{\Lambda,\Delta}\big(zp_\Delta^k +\bar p_\Delta\sum_{j=0}^{k-1} p_\Delta^j\big)=
 \prod_{k=0}^\infty g_{\Lambda,\Delta}\big(1-(1-z)p_\Delta^k \big).
\end{align}
In the factors $g_{\Lambda,\Delta}\big(1-(1-z)p_\Delta^k \big)$ we recognize the expression for $\phi_k(z)$ as in (\ref{pgfk}).\\

\noindent \textbf{First two moments.} We now evaluate the first two moments of $M$. This is an interesting computation in its own right, but it also provides useful results that can be exploited when considering this system under the central limit scaling (as is done in the next subsection).

Differentiating (\ref{fp}) and letting $z \uparrow 1$, we obtain a fixed-point equation, 
\[\phi'(1) = \phi'(1) {\rm e}^{-\mu \Delta} +g'_{\Lambda,\Delta}(1)= \phi'(1) {\rm e}^{-\mu \Delta} +r_\Delta\,{\mathbb E}\Lambda.\]
Hence ${\mathbb E}M = \phi'(1)=r_\Delta\,{\mathbb E}\Lambda/(1-{\rm e}^{-\mu \Delta})={\mathbb E}\Lambda/\mu.$
This quantity could have been computed more directly as well, using a standard identity for conditional means:
\begin{equation}\label{cm} 
\E M = \sum_{k=0}^\infty \E M_k=\sum_{k=0}^\infty \E\big[{\mathbb E}[M_k\,|\,\Lambda_k]\big].
\end{equation}
Then observe that $(M_k\,|\,\Lambda_k)$ is Poisson, and hence its mean equals its parameter. 
As a result, (\ref{cm}) equals
\[ \E M = \sum_{k=0}^\infty \E[\kappa_{k}(\Lambda_{k})]=\E[\Lambda] r_\Delta \sum_{k=0}^\infty p_\Delta^k = \E\Lambda/\mu.\]
For the variance we use that
\[\phi''(1)=\phi''(1)p_{ \Delta}^2+2\phi'(1)p_{\Delta}g_{\Lambda,\Delta}'(1)+g_{\Lambda,\Delta}''(1),\]
and hence
\[\phi''(1) = \frac{2\phi'(1) \,p_\Delta \,g_{\Lambda,\Delta}'(1)}{1-p_\Delta^2}
+\frac{g_{\Lambda,\Delta}''(1)}{1-p^2_{ \Delta}}=2\frac{(\E\Lambda)^2}{\mu^2}\frac{p_\Delta}{1+p_\Delta}+\frac{\E\Lambda^2}{\mu^2}\frac{1-p_\Delta}{1+p_\Delta}.\]
It thus follows that, after some algebra,
\begin{eqnarray}\nonumber
\Var M&=& \phi''(1)+\phi'(1)-(\phi'(1))^2\\
&=&\E\Lambda / \mu + C\, \Var \Lambda / \mu^2,\label{VAR}
\end{eqnarray}
where $C:=({1-p_\Delta})/({1+p_\Delta})$.

Alternatively, one could use the `law of total variance' to identify $\Var M$:
\begin{equation}\label{taga}
\Var M =\sum_{k=0}^\infty \Var M_k 
=\sum_{k=0}^\infty \E[\Var( M_k \,|\,\Lambda_k)] + \sum_{k=0}^\infty \Var(\E[M_k\,|\,\Lambda_k]). 
\end{equation}
Observe that, because of the `mixed Poisson property', $\E[\Var( M_k \,|\,\Lambda_k)] = \E[\kappa_{k}(\Lambda_{k})]$, and as a result the first term at the right-hand side of (\ref{taga}) equals $\E M.$ 
The second term, which is inherently non-negative, gives rise to `overdispersion', i.e., the effect that the variance of the stationary queue length {\it exceeds} the corresponding mean. 
This is a distinguishing feature compared to the analogous system in which the Poissonian arrival rate is deterministic:  the stationary queue length in an M/M/$\infty$ system is Poisson, and cannot accommodate any overdispersion.
In order to evaluate the second term in the right-hand side of (\ref{taga}), we note that 
\begin{equation}\label{tagb}
\Var(\E(M_k\,|\,\Lambda_k)) =\Var \big(\kappa_{k}(\Lambda_{k})\big)=r_\Delta^2 p_\Delta^{2k}\cdot \Var \Lambda.
\end{equation}
Substituting (\ref{tagb}) in the second term in the right-hand side of (\ref{taga}), we find that ${\mathbb V}{\rm ar}\, M$ equals (\ref{VAR}), as desired. 

Formula (\ref{VAR}) lends itself to a nice interpretation: the term $\E \Lambda/\mu$ is the contribution to the variance that one would have if the arrival rate would have had the {\it deterministic} value $\E \Lambda$, whereas the term $C\, \Var \Lambda / \mu^2$ needs to be added in order to deal with the non-Poisson variability due to the stochasticity of the arrival rate.\\

\noindent \textbf{Transient behavior.}
As the analysis of the transient system behavior strongly resembles its stationary counterpart, we restrict ourselves to a short account of this.
We let the system start empty (for ease of presentation; a non-empty initial condition can be analyzed without any additional difficulty). 
Denote by $M(t)$ the number of jobs present at time $t$. 
Then, for $n$ the smallest integer such that $t-n\Delta < \Delta$,
\[M(t) = \sum_{j=0}^{n-1} \bar M_j + \bar{M}_{[n\Delta,t)},\]
where $\bar M_j$ ($\bar{M}_{[n\Delta,t)}$) represents the number of jobs that have arrived between in $[j\Delta,(j+1)\Delta)$ ($[n\Delta,t)$) and are still around at  $n\Delta$ ($t$).
As before, these have \PGF{}s
\begin{eqnarray*}
{\mathbb E}z^{\bar M_j}&=& \E \exp\big(-\Lambda r_{\Delta} p_\Delta^{n-(j+1)}\e^{-\mu(t-n\Delta)}(1-z)\big);\\
{\mathbb E}z^{\bar{M}_{[n\Delta,t)}}&=& \E \exp\big(-\Lambda/\mu (1-\e^{-\mu(t-n\Delta)})(1-z)\big).
\end{eqnarray*}
As the individual random variables $\bar M_1, \bar M_2, \dots $ and $\bar{M}_{[n\Delta,t)}$ are independent, $M(t)$ is mixed Poisson with random parameter 
\begin{align}\label{part}
\kappa_{t}(\Lambda):=\int_0^{t} \Lambda(s) \e^{- \mu s} \, \dd s.
\end{align}
\subsection{Limit results}\label{ss2.2}

This section focuses on the central limit regime that results from simultaneously scaling, in a controlled way, both the arrival rate $\Lambda$ and the sampling frequency $\frac{1}{\Delta}$ as in (\ref{scale}).
Let the scaled counterpart of $\Lambda(t)$ be $N \Lambda^{(N)}(t)$, with
\begin{align} \label{defslab} 
\Lambda^{(N)}(t) := \sum_{j=0}^{\infty} \Lambda_j 1_{[j \Delta N^{-\alpha}, (j+1)\Delta N^{-\alpha})}(t).
\end{align}
That is, the sampling frequency and the arrival rates are both inflated as we let $N$ tend to $\infty$, but, importantly,  at rates that are not necessarily identical. 
As mentioned in the introduction, depending on the value of $\alpha$, we obtain fundamentally different behavior.

We consider a sequence of systems indexed by $N$, where the $N$-scaled system uses a mixed Poisson arrival process with time-dependent random rate $N\Lambda^{(N)}(t)$.
Let $M^{(N)}$ denote the stationary queue length in the $N$-scaled system, with parameter $N\kappa(\Lambda^{(N)})$ (cf. (\ref{pars})).
We start our exposition by a preliminary calculation, in which we compute the mean and variance of $M^{(N)}$ and study their behavior for large $N$, which indeed reveals the announced trichotomy. 
Then, after centering and normalizing $M^{(N)}$, we derive a central limit theorem. \\
  
\noindent \textbf{Qualitative behavior of first two moments: trichotomy in variance.}
First, we identify the steady-state mean and variance in our scaling regime, using (\ref{cm}), (\ref{taga}) and (\ref{tagb}). 
We find that
\begin{align}
\label{exp} \E M^{(N)} &= N\E\Lambda / \mu;\\
\label{var} \Var M^{(N)} &= N \E\Lambda / \mu + N^2\frac{1-{\rm e}^{-\mu\Delta N^{-\alpha}}}{1+{\rm e}^{-\mu\Delta N^{-\alpha}}}\Var\Lambda / \mu^2,
\end{align}
where it is noted that for large $N$, (\ref{var}) behaves approximately as
\[
N \E\Lambda / \mu + N^{2-\alpha} \Delta \Var \Lambda/(2\mu)
\]
(the ratio of the two converges to $1$).
We thus observe the trichotomy
\begin{align} \label{dich}
{\mathbb V}{\rm ar}\, M^{(N)} \sim 
\begin{cases} 
N{\displaystyle {\mathbb E}\Lambda / \mu }&\text{ if }\alpha>1;\\
N^{2-\alpha} \Delta{\mathbb V}{\rm ar}\,\Lambda /(2\mu) &\text{ if }\alpha<1;\\
N \big({\mathbb E}\Lambda /\mu + \Delta{\mathbb V}{\rm ar}\,\Lambda / (2\mu)\big)&\text{ if }\alpha=1.
\end{cases}
\end{align}
For $\alpha>1$, the sampling frequency dominates the variability of $\Lambda$. 
Consequently, the model behaves essentially as an M/M/$\infty$ system, with the variance of $M^{(N)}$ being linear in $N$ and equal to $\E M^{(N)}$, for large $N$.
For $\alpha<1$, we find a superlinear relation between $N$ and ${\mathbb V}{\rm ar}\, \Lambda$, and both the sampling frequency (i.e., the reciprocal of the interval length $\Delta$) and the variance of $\Lambda$ play a role.
Hence, the asymptotic variance indeed grows faster than the asymptotic mean for $\alpha <1$; in this regime the system is overdispersed.
For $\alpha=1$, the variance is `slightly larger' than for $\alpha>1$, but it is still linear in $N$.
In this case the sampling frequency and the variance of $\Lambda$ grow at the same rate, so that the variance for $M^{(N)}$ combines the effects observed in the two former cases.

As observed from the above computation, the variance of $M^{(N)}$ is essentially proportional to $N^\gamma$ with $\gamma:=\max\{1,2-\alpha\}$.
As a consequence, one may expect that, under (\ref{scale}),
\begin{equation}\label{clt}
\check M^{(N)}:=
N^{-\gamma/2}(M^{(N) }- {\mathbb E}\,M^{(N)})
\end{equation}
converges to a (zero-mean) normally distributed random variable. 
It is this property that we verify now.\\

\noindent \textbf{Asymptotic normality.}
We show how to establish asymptotic normality for the centered and normalized version of $M^{(N)}$ in (\ref{clt}) via evaluation of the corresponding Laplace transform.
Appealing to L\'evy's convergence theorem, we establish the desired convergence in distribution.
For simplicity,  the proof of Thm.\ \ref{CLT} assumes that all moments of $\Lambda$ are finite; however, as will appear from the proof of Thm.\ \ref{FCLT} only finiteness of the first two moments is necessary.

\begin{thm}[{\sc clt}] \label{CLT}
As $N \rightarrow \infty$, $\check M^{(N)}$ converges to a zero-mean normally distributed random variable with variance \[\sigma^2:=\frac{\E \Lambda}{\mu}\,1_{\{\alpha \geqslant  1\}} + \frac{\Delta{\mathbb V}{\rm ar}\,\Lambda}{2\mu}\,1_{\{\alpha\le 1\}}.\]
\end{thm}
\begin{proof}
Let $\phi^{(N)}(z)$ be the counterpart of (\ref{pgf}) under scaling as in (\ref{scale}); likewise $g^{(N)}_{\Lambda,\Delta}(z)$ is the counterpart of $g_{\Lambda,\Delta}(z)$. 
Then
\[
\phi^{(N)}(z) = \prod_{k=0}^\infty g^{(N)}_{\Lambda,\Delta}\big(1-(1-z)\e^{-\mu kN^{-\alpha}\Delta}\big).
\]
We are interested in the behavior of $M^{(N)}$ in the central limit regime, hence we need to analyze the limiting distribution 
of $\check M^{(N)}$. 
To this end, we evaluate the logarithm of the corresponding Laplace transform:
\begin{equation}\label{c2}
\log{\mathbb E}\exp\big(-s N^{-\gamma/2}(M^{(N)} - \E M^{(N)})\big)=s N^{1-\gamma/2} \E\Lambda / \mu+\log\phi^{(N)}\big(\e^{-s N^{-\gamma/2}}\big).
\end{equation}
We now use that $\log\phi^{(N)}(\e^{-s N^{-\gamma/2}})$ equals
\begin{align}\label{cumul}
\sum_{k=0}^\infty \log{\mathbb E}\e^{- N\Lambda/\mu\big(1-{\rm e}^{-\mu N^{-\alpha}\Delta}\big)\big(1-{\rm e}^{-s N^{-\gamma/2}}\big){\rm e}^{-\mu k N^{-\alpha}\Delta}}.
\end{align}
Observe that (\ref{cumul}) is the sum of cumulant generating functions (which is again a cumulant generating function), each of them related to the random variable $\Lambda$ but evaluated at different arguments.
Let $m_\ell$ denote the $\ell$-th cumulant of $\Lambda$ (for $\ell\in{\mathbb N}$); in particular $m_1={\mathbb E}\Lambda$ and $m_2 ={\mathbb V}{\rm ar}\, \Lambda$.
In addition, we define
\[\zeta_k^{(N)}(s):= -N / \mu\big(1-{\rm e}^{-\mu N^{-\alpha}\Delta}\big)\big(1-{\rm e}^{-s N^{-\gamma/2}}\big){\rm e}^{-\mu k N^{-\alpha}\Delta}.\]
Then it follows that
\begin{align} \label{cumsum}
\log\phi^{(N)}\big({\rm e}^{-s N^{-\gamma/2}}\big) = \sum_{\ell =1}^\infty \frac{m_\ell}{\ell!}\sum_{k=0}^\infty    \big(\zeta_k^{(N)}(s)\big)^\ell.
\end{align}
Let us first consider the contribution of the term corresponding to $\ell =1$. 
Observe that, as $N \to \infty$,
\begin{align}
m_1\,\sum_{k=0}^\infty    \zeta_k^{(N)}(s) = - N\,{\mathbb E}\Lambda / \mu\big(1-{e}^{-s N^{-\gamma/2}}\big)\sim 
-s N^{1-\gamma/2} {\mathbb E}\Lambda / \mu
+\frac12 s^2 N^{1-\gamma }{\mathbb E}\Lambda / \mu.\label{c1}
\end{align} Note that the first term in the right-hand side of (\ref{c1}) is canceled by the first term in the right-hand side of (\ref{c2}), so that we are left with the second term, i.e.,
\begin{equation}\label{c11}
\frac12 s^2 N^{1-\gamma }{\mathbb E}\Lambda / \mu.
\end{equation} 
The second term in (\ref{cumsum}) corresponding to $\ell =2$ gives
\begin{eqnarray}\label{c3}
\qquad \frac{1-\e^{-\mu N^{-\alpha}\Delta}}{1+\e^{-\mu N^{-\alpha}\Delta}}\frac{(1-\e^{-s N^{-\gamma/2}})^2}{2\mu^2} N^2\,\Var\Lambda \sim
\frac12 s^2 N^{2-\alpha-\gamma}\, \Delta\,\Var\Lambda / (2\mu).
\end{eqnarray}
Now compare the asymptotic expansion identified in (\ref{c11}) and (\ref{c3}).
In case $\alpha >1$, we have that $\gamma=1$, so that (\ref{c11}) equals $\frac{1}{2}\,s^2\, {\E\Lambda}/\mu$, whereas (\ref{c3}) converges to zero. 
On the other hand, for $\alpha<1$ we have $\gamma=2-\alpha$, and hence (\ref{c11}) converges to zero, whereas (\ref{c3})  behaves as $\frac{1}{2}s^2{\Delta\,\Var\Lambda}/(2\mu)$. 
Finally, if $\alpha=1$, we find that both terms converge to the expected finite positive limit.

We now check that the terms in (\ref{cumsum}) for $\ell\geqslant 3$ vanish as $N \to \infty$.
For large $N$ the terms can be approximated as follows,
\[\sum_{k=0}^\infty    \big(\zeta_k^{(N)}(s)\big)^\ell\sim N^\ell \frac{\big(1-{e}^{-\mu N^{-\alpha}\Delta}\big)^\ell}{1-{e}^{-\ell\mu N^{-\alpha}\Delta}}\big(1-{e}^{-s N^{-\gamma/2}}\big)^\ell  \sim N^\ell \,  \frac{\mu^\ell N^{-\alpha\ell}{\Delta^\ell}}{\ell\mu N^{-\alpha}\Delta}\frac{s^\ell}{N^{\gamma\ell/2}},\]
hence being of order $N^\delta$ with $\delta=\delta(\alpha):=\ell(1-\alpha-\gamma/2)+\alpha$. 
In case $\alpha\geqslant 1$, we get (bearing in mind that $\gamma=1$ and $\ell \geqslant 1$)
\[\delta=\ell(\frac12 -\alpha) + \alpha=\frac12 \ell + \alpha (1-\ell)\leqslant\frac12 \ell + 1- \ell=1- \frac{\ell}{2};\] 
on the other hand, in case $\alpha < 1$ we get $\delta = -\ell\alpha/2 +\alpha=\alpha (1-{\ell}/{2})$ (with $\gamma=2-\alpha$).
We conclude that $\delta< 0$ for $\ell\geqslant 3$ and the corresponding terms in (\ref{cumsum}) can indeed be neglected.
We have therefore established that, as $N\to\infty$,
\begin{align*}
\log \E\exp\big(-s N^{-\gamma/2}(M^{(N)} - {\mathbb E} M^{(N)}) \big) \to \frac{1}{2}\sigma^2 \,s^2,
\end{align*}
as claimed. \end{proof}

\section{Functional central limit theorem} \label{sec3}
In this section we generalize the central limit result of Thm.\ \ref{CLT} in two ways. 
First, we establish the functional version: the centered and normalized transient queue length process converges to a limiting process of Ornstein-Uhlenbeck type with parameters that depend on the value of $\alpha$. 
Second, we extend this to a multidimensional setting with correlated arrivals: every arrival triggers jobs in multiple queues. 
The correlation structure of the resulting multidimensional Gaussian limiting process is explicitly identified. 

Let us start by describing the mechanics of the generalized setting. 
We consider a parallel system in which $d$ queues are fed by a {\it single} arrival process that was constructed in the same way as the one in the previous section: a Markovian process with arrival rate $\Lambda(t)$ as in (\ref{deflab}).
The service times in queue $i$ are i.i.d.\ exponential random variables with mean $\mu_i^{-1}$; the service processes  of the individual queues are independent, and also independent of the arrival process.
We perform the same scaling as before: the sampling frequency is sped up by a factor $N^\alpha$, while the (random) arrival rate is blown up by a factor $N$.
This results in a mixed Poisson arrival process with time-dependent rate $\Lambda^{(N)}(t)$ as in (\ref{defslab}).
Let \[\M^{(N)}(t)=(M^{(N)}_1(t), \dots, M^{(N)}_d(t))^{\rm T},\] where $M^{(N)}_i(t)$ is the queue length at time $t$ in the $i$-th queue of the $N$-scaled system, for $i\in \{1,\dots,d\}$.
Note that the $M^{(N)}_i(t)$ are mixed Poisson with time-dependent random parameter $N \kappa_{t,i}(\Lambda^{(N)})$, with $\kappa_{t,i}(\Lambda^{(N)})$ as defined in (\ref{part}) but with $\mu$ replaced by $\mu_i$.

We now present an alternative way of writing $M^{(N)}_i(t)$, which facilitates the use of a martingale central limit theorem. 
We introduce the functional
\[\Psi[X](t):=\int_0^t X(s)\, \mathrm{d}s,\]
mapping the stochastic process $\{X(s): s\in[0,t]\}$ to a real number; then
$\mu_i\,\Psi[M_i^{(N)}](t)$ is to be interpreted as the `cumulative service capacity' in queue $i$ over the interval $[0,t]$.
In addition, for the `cumulative arrival rate' we have $\Psi[\Lambda](t)$,
with scaled counterpart $N \Psi[\Lambda^{(N)}](t)$.
By the law of large numbers, $\Psi[\Lambda](t)/t$ converges a.s.\ to $\E \Lambda$ as $t \to \infty$ and for fixed $t$, $\Psi[\Lambda^{(N)}](t)$ converges a.s.\ to $t\, \E \Lambda$ as $N \to \infty$.

With $Y_0(\cdot),\ldots,Y_d(\cdot)$ denoting independent unit-rate Poisson processes,
\begin{align} \label{ql}
M_i^{(N)}(t)\overset{\mathrm{d}}{=} M_i^{(N)}(0)+ Y_0\big(N \Psi[\Lambda^{(N)}](t)\big)-Y_i\big(\mu_i\,\Psi[M_i^{(N)}](t)\big).\end{align}
Our objective is to derive a $d$-dimensional \FCLT{} for $\M^{(N)}(\cdot)$. 
This result characterizes the time-dependent queue length in the scaled system and makes explicit how the correlated arrivals lead to correlation between the individual queue length processes.
It will be stated and proved in subsection \ref{3.2}; first we study the stationary behavior by presenting the corresponding first two moments of $\M^{(N)}$ (including covariances between the individual queues). 

\subsection{Qualitative behavior of first two moments in stationarity}

Note that the individual queue lengths are only coupled through the arrival process, so under  (\ref{scale}), the mean and variance of $\M^{(N)}$ are, as in (\ref{exp}) and (\ref{var}), given by
\begin{align*}
\E M_i^{(N)} &= N \E \Lambda / \mu_i,\\
\Var M^{(N)}_i&=N \E \Lambda / \mu_i+N^2 \Var \Lambda / \mu_i^2 \frac{1-p_i(\Delta N^{-\alpha})}{1+p_i(\Delta N^{-\alpha})} \sim N \E\Lambda / \mu_i + N^{2-\alpha} \Delta \Var\Lambda / (2\mu_i),
\end{align*}
for $i = 1,\dots,d$.
Hence, we find the same behavior as in (\ref{dich}). 
Interestingly, the same trichotomy is observed for the covariances, as stated in the next lemma.
\begin{lemma}[Covariance in $\M^{(N)}$]\label{lem:gedrag}
For $i,k\in \{1, \dots, d\}$ with $i\neq k$, and for large $N$,
\begin{equation}\label{gedrag}
\Cov(M^{(N)}_i, M^{(N)}_k)\sim \begin{cases}
N \E\Lambda /(\mu_i+\mu_k)&\text{ if }\alpha>1;\\
N^{2-\alpha} \Delta\Var(\Lambda) / (\mu_i+\mu_k)&\text{ if }\alpha<1;\\
 N\big(\E\Lambda /(\mu_i+\mu_k)+\Delta\Var(\Lambda) / (\mu_i+\mu_k)\big)
&\text{ if }\alpha=1.
\end{cases}\end{equation}
\end{lemma}
\begin{proof}
Without loss of generality, we take $i=1$ and $k=2$.
We first consider the non-scaled model, by studying the joint probability generating function,
\[
\E z_1^{M_1(n\Delta)}z_2^{M_2(n\Delta)}=\prod_{j=0}^{n-1} \xi_{jn}(z_1,z_2),
\]
where $\xi_{jn}(z_1,z_2)$ is the contribution due to the slot between $j\Delta$ and $(j+1)\Delta$; as $z_1$ and $z_2$ are held fixed for the moment, we suppress them.
Now we introduce functions (for $\ell=1,2$) 
\[
f_\ell(r,n):=\e^{-\mu_\ell(n\Delta-r)},\:\:\:
g_j(\mu,n):=\tfrac{1}{\mu \Delta}(1-\e^{-\mu\Delta})\, \e^{-\mu(n-j)\Delta},
\]
where it is noted that $g_j(\mu,n)$ behaves as $\e^{-\mu(n-j)\Delta}$ for small $\Delta$.
In addition, we define the quantities
\begin{align*} 
\zeta^{++}_{jn}&:= \int_{j\Delta}^{(j+1)\Delta} \tfrac{1}{\Delta} f_1(r,n)f_2(r,n){\rm d}r=g_j(\mu_1+\mu_2,n),\\
\zeta^{+-}_{jn}&:=\int^{(j+1)\Delta}_{j\Delta} \tfrac{1}{\Delta} f_1(r,n)(1-f_2(r,n)){\rm d}r=g_j(\mu_1,n)-g_j(\mu_1+\mu_2,n),\\
\zeta^{-+}_{jn}&:=\int^{(j+1)\Delta}_{j\Delta} \tfrac{1}{\Delta}(1- f_1(r,n))f_2(r,n){\rm d}r=g_j(\mu_2,n)-g_j(\mu_1+\mu_2,n),\\
\zeta^{--}_{jn}&:=\int^{(j+1)\Delta}_{j\Delta} \tfrac{1}{\Delta} (1-f_1(r,n))(1-f_2(r,n)){\rm d}r=1-g_j(\mu_1,n)-g_j(\mu_2,n)+g_j(\mu_1+\mu_2,n).
\end{align*}
Using arguments similar to those we have used before,
\begin{eqnarray*}
\xi_{jn}&:=&{\mathbb E}\big(\sum_{m=0}^\infty \e^{-\lambda\Delta}\frac{(\lambda\Delta)^m}{m!}\big(\zeta^{++}_{jn}z_1z_2+\zeta^{+-}_{jn}{z_1}+\zeta^{-+}_{jn}z_2+\zeta^{--}_{jn} \big)^m\big)\\
&=&{\mathbb E}\exp\big(\Lambda\Delta\big(\zeta^{++}_{jn}z_1z_2+\zeta^{+-}_{jn}z_1+\zeta^{-+}_{jn}z_2+\zeta^{--}_{jn}-1\big)\big)\\
&\sim&{\mathbb E}\exp\big(\Lambda\Delta\big(\prod_{i=1}^2 \big((z_i-1){\e}^{-\mu_i (n-j) \Delta}+1\big)-1\big)\big) \text{ for small }\Delta.
\end{eqnarray*}
Because the contributions to $M_1(n\Delta)$ and $M_2(n\Delta)$ resulting from different time intervals are independent, we obtain that 
\begin{align*}
\Cov\big(M_1(n\Delta),M_2(n\Delta) \big)
=\sum_{j=0}^{n-1}\left. \big( \frac{\partial^2}{\partial z_1\partial z_2}
\xi_{jn}(z_1,z_2) -\frac{\partial}{\partial z_1}
\xi_{jn}(z_1,z_2) \frac{\partial}{\partial z_2}
\xi_{jn}(z_1,z_2)\big)\right|_{z_1\uparrow 1,z_2\uparrow 1} .
\end{align*}
Now imposing scaling (\ref{scale}) and considering the stationary behavior by letting $n\to\infty$, it is readily derived that (for large $N$)
\[
\E \,z_1^{M^{(N)}_1}z_2^{ M^{(N)}_2}\sim \prod_{j=0}^\infty \E\exp\big(\Lambda \Delta N^{1-\alpha}\big(\prod_{\ell=1}^2\big((z_{\ell}-1)e^{-\mu_{\ell} j\Delta/N^\alpha}+1\big)-1\big)\big);
\]
observe that, for reasons of symmetry, it is allowed to replace $n-j$ by $j$ in the definition of the $g_j(\mu,n)$. 
We thus arrive at
\begin{align*}
\Cov&\big(M_1^{(N)},M_2^{(N)} \big)\\
&\sim \sum_{j=0}^\infty \big( (\E\Lambda\, \Delta N^{1-\alpha} + {\mathbb E}[\Lambda^2]\Delta^2 N^{2-2\alpha}  )\e^{-(\mu_1+\mu_2)j\Delta/N^\alpha}- \prod_{\ell=1}^2{\mathbb E}\Lambda\,\Delta N^{1-\alpha} \e^{-\mu_\ell j\Delta/N^\alpha}\big)\\
&=\frac{{\mathbb E}\Lambda\, \Delta N^{1-\alpha}+ {\mathbb V}\mathrm{ar}(\Lambda) \,\Delta^2 N^{2-2\alpha}}{1-{\e}^{-(\mu_1+\mu_2)\Delta N^{-\alpha}}},
\end{align*}
 which behaves in accordance with (\ref{gedrag}) for $N$ large.
\end{proof}

Recall that $\gamma =\max\{1,2-\alpha\}$; the above computation shows that the covariance matrix of ${\boldsymbol M}^{(N)}$ is essentially proportional to $N^\gamma$. 
Therefore, we expect that the centered and normalized version of the joint stationary queue length process converges to a (zero-mean) $d$-dimensional Gaussian random vector with covariance matrix $C$ such that
\begin{align*}
C_{i k} = \begin{cases}
1_{\{\alpha \leqslant1\}}\mathbb{E}\Lambda / \mu_i+ 1_{\{\alpha \geqslant1\}}\Delta {\mathbb V}{\rm ar}(\Lambda) /(2\mu_i) &\text{ if } i = k,\\
1_{\{\alpha \leqslant1\}}\mathbb{E}\Lambda /(\mu_i + \mu_k) + 1_{\{\alpha \geqslant1\}} \Delta {\mathbb V}{\rm ar}(\Lambda)/(\mu_i + \mu_k) &\text{ if } i \neq k,
\end{cases}
\end{align*}
for $i,k \in \{1,\dots,d\}$. 
This is verified in the next subsection.

\subsection{Proof of functional central limit theorem based on \MCLT} \label{3.2}

The main objective of this subsection is to derive a functional limit theorem for ${\boldsymbol M}^{(N)}(t)$, the vector describing the queue lengths of the scaled system at time $t$.
To this end, we consider the process $\tilde M_i^{(N)}(\cdot):= M_i^{(N)}(\cdot)/N$, for which we have
\begin{equation}\label{SDE}
\tilde M_i^{(N)}(t) = \tilde M_i^{(N)}(0) + N^{-1} Y_0\big(N \Psi[\Lambda^{(N)}](t)\big)-N^{-1} Y_i\big(N\mu_i\,\Psi[\tilde{M}^{(N)}](t)\big).
\end{equation}
We will need the following lemma, which uses the law of large numbers for Poisson processes; see \cite{ABM14}.
\begin{lemma} \label{lem2}
Let $Y$ be a unit-rate Poisson process.
Then for any $U>0$, almost surely
\[
\lim_{N \to \infty}\sup_{0\leqslant u \leqslant U} \left|\frac{Y(Nu)}{N}-u\right| = 0.
\]
\end{lemma}
The uniform convergence in Lemma \ref{lem2} entails that (\ref{SDE}) converges almost surely to the solution of the functional equation
\begin{align}\label{fl}
\varrho_i(t)=\varrho_{i,0} + t\,\E\Lambda - \mu_i\,\Psi [\varrho_i](t),
\end{align}
as $N \to \infty$, under the proviso  that $\tilde M^{(N)}_i(0)$ converges a.s.\ to some value $\varrho_{i,0}$ for $i=1,\dots, d$. 
The solution is given by a convex mixture of the initial position $\varrho_i(0)=\varrho_{i,0}$ and the limiting value $\E\Lambda/\mu_i$:
\begin{align}\label{expl}
\varrho_i(t)=\varrho_{i,0} \e^{-\mu_i t} + \frac{\E\Lambda}{\mu_i}\big(1-\e^{-\mu_i t}\big).
\end{align}

Having identified this {\it fluid limit}, the next objective is to establish an \FCLT{} for the centered and normalized process ${\boldsymbol U}^{(N)}(\cdot)$ given by
\begin{align}\label{fcltproc}
U_i^{(N)}(t):= N^\frac{\beta}{2}\big(\tilde M_i^{(N)}(t)-\varrho_i(t)\big), 
\end{align}
with $\beta:=2-\gamma=\min\{1,\alpha\}$. 
Here we closely follow the approach in \cite{ABM14}, where the idea is to use an {\sc mclt}, so as to obtain weak convergence to a (generalized) Ornstein-Uhlenbeck process.
The version of the {\sc mclt}{} that we need in our setting is stated below.
\begin{thm}[{\sc mclt}, \cite{ABM14}] \label{MCLT} 
Let $\{{\boldsymbol R}^{(N)}\}_{N \in \mathbb{N}}$ be a sequence of $\mathbb{R}^d$-valued martingales. 
Assume that the following condition on the jump sizes is met:
\begin{align}\label{jump}
\lim_{N \to \infty} \E \big[ \sup_{s \leqslant t} \big|{\boldsymbol R}^{(N)}(s)-{\boldsymbol R}^{(N)}(s^{-})\big|\big]=0 ;
\end{align}
in addition, assume that, as $N \to \infty$,
\[
\big[R_i^{(N)},R_k^{(N)}\big]_t \to C_{ik}(t)
\]
for a deterministic function $C_{ik}(t)$, continuous in $t$ for all $t>0$ and for $i,k=1, \dots, d$.
Then the process ${\boldsymbol R}^{(N)}$ converges weakly to a centered Gaussian process ${\boldsymbol W}$ with independent increments whose covariance matrix is characterized by
\[\E\big[W_{i}(t) \cdot W_k(t)^{\rm T}\big]=C_{ik}(t).\]
\end{thm}

Introducing compensated unit-rate Poisson processes $\tilde Y_i(t):=Y_i(t)-t$, we define
\begin{align}\label{processes}
&\check {\boldsymbol Y}_{0}^{(N)}(t):=N^{\frac{\beta}{2} -1}\begin{pmatrix}\tilde Y_0\big(N\Psi[\Lambda^{(N)}](t)\big)\\ \vdots \\ \tilde Y_0\big(N\Psi[\Lambda^{(N)}](t)\big)\end{pmatrix} ,\\
&\nonumber \,\\
&\check{\boldsymbol Y}^{(N)}(t):=N^{\frac{\beta}{2} -1}\begin{pmatrix}\tilde Y_1\big(N\mu_1\,\Psi[\tilde M_1^{(N)}](t)\big)\\ \vdots \\ \tilde Y_d\big(N\mu_d\,\Psi[\tilde M_d^{(N)}](t)\big)\end{pmatrix}.
\end{align}
\begin{lemma}\label{term3}
Consider the $d$-dimensional processes $\check {\boldsymbol Y}_{0}^{(N)}(\cdot)$ and $\check{\boldsymbol Y}^{(N)}(\cdot)$.
If $\alpha\geqslant 1$, then as $N \to \infty$ these processes converge weakly to $d$-dimensional zero-mean Brownian motions with covariance matrices $K_0(t) := (t\,{\mathbb E}\Lambda){\boldsymbol 1}{\boldsymbol 1}^{\rm T} $ and
$K(t):={\rm diag}\{\mu_1\,\Psi[\varrho_1](t),\ldots,\mu_d\,\Psi[\varrho_d](t)\}$, respectively; if $\alpha<1$ the limiting covariance matrices equal $\boldsymbol{0}$.\end{lemma}
\begin{proof}
We start by checking the conditions of Thm.\ \ref{MCLT}.
First, observe that for each $N$, $\check {\boldsymbol Y}_{0}^{(N)}(\cdot)$ and $\check{\boldsymbol Y}^{(N)}(\cdot)$ are $d$-dimensional real-valued martingales.
Also, condition (\ref{jump}) is met, as both for ${\boldsymbol R}^{(N)}=\check {\boldsymbol Y}_{0}^{(N)}$ and ${\boldsymbol R}^{(N)}=\check {\boldsymbol Y}^{(N)}$,
\[\lim_{N \to \infty} \E \big[ \sup_{s \leqslant t} \left|{\boldsymbol R}^{(N)}-{\boldsymbol R}^{(N)}(s^{-})\right|\big]<\infty,\]
whereas $N^{\frac{\beta}{2}-1}\leqslant1/{\sqrt{N}}$ converges to zero.

Note that $\beta=\min\{1,\alpha\}$, so that $\beta -2=\min\{-1,\alpha - 2\}$.
Now observe that for $\alpha \geqslant1$ (and hence $\beta-2 = -1$) the quadratic covariation of $\check {\boldsymbol Y}_{0}^{(N)}(\cdot)$,
\[
\big[N^{\frac{\beta}{2}-1} \tilde Y_0\big(N\Psi[\Lambda^{(N)}](t)\big)\big]_t = N^{\beta-2} Y_0\big(N\Psi[\Lambda^{(N)}](t)\big),
\]
converges to $t\,{\mathbb E}\Lambda$ as $N \to \infty$ ($0$ for $\alpha < 1$), by virtue of Lemma \ref{lem2}. 
The covariance matrix for $\check{\boldsymbol Y}^{(N)}(\cdot)$ is determined in the same way; for $\alpha \geqslant 1$ the diagonal entries are given by
\begin{align*}
\lim_{N\to\infty}\big[N^{\frac{\beta}{2} -1}\tilde Y_i\big(N\mu_i\,\Psi[\tilde M_i^{(N)}](t)\big)\big]_t &= \lim_{N\to\infty}N^{\beta-2}Y_i\big(N\mu_i\,\Psi[\tilde M^{(N)}](t)\big)= \mu_i\,\Psi[\varrho_i](t)
\end{align*}
(which would equal $0$ for $\alpha<1$),
whereas for $i \neq k$ (then $\tilde Y_i(\cdot)$ and $\tilde Y_k(\cdot)$ are independent)
\[
\lim_{N\to\infty}\big[N^{\frac{\beta}{2}-1}\tilde Y_i\big(N\mu_i\,\Psi[\tilde M_i^{(N)}](t)\big),N^{\frac{\beta}{2} -1}\tilde Y_k\big(N\mu_k\,\Psi[\tilde M_k^{(N)}](t)\big)\big]_t =0,
\]
with $i,k\in \{1,\dots,d\}$.
For $\alpha \geqslant 1$, Thm.\ \ref{MCLT} yields that the processes converge weakly to $d$-dimensional Brownian motions with covariance matrices $K_0(t)$ and $K(t)$.
On the other hand, for $\alpha < 1$ the entries of the covariance matrices all vanish as $N \to \infty$. 
As a result, both $\check {\boldsymbol Y}_{0}^{(N)}(\cdot)$ and $\check{\boldsymbol Y}^{(N)}(\cdot)$ converge to a process identical to ${\boldsymbol 0}$.
\end{proof}

Stated below is the main theorem of this section: an \FCLT{} for ${\boldsymbol U}^{(N)}(\cdot)$, the process defined via (\ref{fcltproc}).
In line with earlier findings, three regimes need to be distinguished: $\alpha>1$ (the fast regime), $\alpha<1$ (the slow regime) and $\alpha=1$ (the intermediate regime).

\begin{thm}[{\sc fclt}] \label{FCLT}
As $N \rightarrow \infty$, ${\boldsymbol U}^{(N)}(\cdot)$ converges weakly to a zero-mean $d$-dimensional Gaussian process with covariance matrix given by
\begin{align}\label{cov}
C_{ii}(t)&:=
1_{\{\alpha \geqslant1\}}
\big( \E \Lambda / \mu_i+\varrho_{i,0} \e^{-\mu_i t}\big)(1-\e^{-\mu_i t}) 
+ 1_{\{\alpha \leqslant1\}}
\Delta\Var\,\Lambda /(2 \mu_i)(1-\e^{-2\mu_i t}),\\ \label{covv}
C_{ik}(t)&:=\big(
1_{\{\alpha \geqslant1\}}
\E \Lambda /(\mu_i+\mu_k)
+
1_{\{\alpha \leqslant1\}}
\Delta {\Var\,\Lambda} /(\mu_i+\mu_k)
\big)
\cdot (1-\e^{-(\mu_i+\mu_k)t}),
\end{align}
for $i\neq k$ ($i,k \in \{1,\dots,d\}$).
\end{thm}

\begin{proof}
Using (\ref{SDE}), we write
\begin{equation}\label{U}
U_i^{(N)}(t)=
N^{\frac{\beta}{2}}\big(\tilde M_i^{(N)}(0) + N^{-1} Y_0\big(N \Psi[\Lambda^{(N)}](t)\big)- N^{-1} Y_i\big(N\mu_i\,\Psi[\tilde M_i^{(N)}](t)\big)- \varrho_i(t) \big),
\end{equation}
for $i=1, \dots, d$.
Adding and subtracting $\varrho_{i,0}$, (\ref{U}) is equivalent to
\begin{align*}
U_i^{(N)}(t)=
&N^{\frac{\beta}{2}}\big(\tilde M_i^{(N)}(0)-\varrho_{i,0}\big) - N^{\frac{\beta}{2}}\big(\varrho_i(t)-\varrho_{i,0}\big) \\
&+N^{\frac{\beta}{2}-1}\big(Y_0\big(N \Psi[\Lambda^{(N)}](t)\big)-Y_i\big(N\mu_i\,\Psi[\tilde M_i^{(N)}](t)\big)\big),
\end{align*}
which, by filling out the implicit form of $\varrho_i(t)$ as in (\ref{fl}), simplifies to 
\[
U_i^{(N)}(t) = U_{0,i}^{(N)}(t) + U_{1}^{(N)}(t) +  U_{2,i}^{(N)}(t),
\]
with
\begin{align*}
U_{0,i}^{(N)}(t) &:= U_i^{(N)}(0)-\mu_i\,\Psi[U_i^{(N)}](t),\\
U_{1}^{(N)}(t) &:=N^{\frac{\beta}{2}}\big(\Psi[\Lambda^{(N)}](t)-t\,\E\Lambda\big),\\
U_{2,i}^{(N)}(t)&:=N^{\frac{\beta}{2} -1}\big(\tilde{Y}_0\big(N\Psi[\Lambda^{(N)}](t)\big)- \tilde Y_i\big(N\mu_i\,\Psi[\tilde M_i^{(N)}](t)\big)\big).
\end{align*}
We consider the three individual components separately.
\begin{enumerate}[(i)]
\item
Component $U_{0,i}^{(N)}(t)$ consists of the starting value of the process, which is assumed to converge to some value $U_i(0)$, minus a reverting term.
It is now straightforward that, as $N \to \infty,$  $U_{0,i}^{(N)}(\cdot)$
converges to $U_{0,i}(t)=U_i(0)-\mu_i\,\Psi[U_{i}](t)$.
\item
Then  consider $U_{1}^{(N)}(\cdot)$.
For $\alpha\le 1$ (and hence $\frac{\beta}{2} = \frac{\alpha}{2}$), the standard functional central limit theorem for partial sums of i.i.d.\ random variables entails that, 
 as $N \to \infty,$
\[
U_{1}^{(N)}(\cdot) \to \sqrt{\Delta\,{\mathbb V}{\rm ar}\,\Lambda} \cdot V(\cdot),
\]
with $V(\cdot)$ a standard Brownian motion.
On the other hand, for $\alpha>1$ the limiting process is identical to ${\boldsymbol 0}$, as a consequence of $\frac{\beta}{2}=\frac12< \frac{\alpha}{2}$.
\item
Finally, from Lemma \ref{term3}, we conclude that ${\boldsymbol U}_{2}^{(N)}(\cdot)$ converges weakly to a  $d$-dimensional zero-mean Brownian motion with covariance matrix $K_0(t)+K(t)$ for $\alpha\geqslant 1$, and to ${\boldsymbol 0}$ else.
\end{enumerate}
Using the above observations, we can now complete the proof.
Each of the three regimes will be considered separately. 

\vspace{2mm}

\noindent {\it 1.~Fast regime $(\alpha >1)$.}
We have obtained above that ${\boldsymbol U}^{(N)}(\cdot)$ converges weakly to the solution ${\boldsymbol U}(\cdot)$ of the $d$-dimensional stochastic integral equation given by
\[U_i(t)=U_i(0)-\mu_i\,\Psi[U_i](t)+ W_i(t\, \E \Lambda+\mu_i\,\Psi[\varrho_i](t))\, \text{ for }i=1, \dots, d\]
with $W_1(\cdot),\ldots,W_d(\cdot)$ standard Brownian motions (but not  independent), or equivalently 
\[U_i(t)=U_i(0)-\mu_i\,\Psi[U_i](t)+\tilde W_0(t\, \E \Lambda) + \tilde W_i\big(\mu_i\,\Psi[\varrho_i](t)\big) \]
with $\tilde W_0(\cdot),\tilde W_1(\cdot),\ldots,\tilde W_d(\cdot)$ independent standard Brownian motions.
It takes a routine calculation to derive that
\[
U_i(t)={\e}^{-\mu_i t}\big( U_i(0)+  \int_0^t \sqrt{\E \Lambda+\mu_i \varrho_i(s)}\,{\e}^{\mu_i s}\, \mathrm{d} W_i(s)\big).\]
All linear combinations of the $U_i(\cdot)$ are Gaussian processes, so we conclude that this $d$-dimensional process is Gaussian. 
It is readily seen that $\E \,U_i(t)=U_i(0) {\e}^{-\mu_i t}$. 
For the variance, an elementary computation gives
\[
{\mathbb V}{\rm ar}\,U_i(t)={\e}^{-2\mu_i t}\big(\int_0^t \big(\E \Lambda+\mu_i \varrho_i(s)\big) \e^{2 \mu_i s} \, \mathrm{d}s\big)
=\big(\E \Lambda/\mu_i+\varrho_{i,0}{\e}^{-\mu_i t}\big)(1-\e^{-\mu_i t}).
\]
Likewise, for the covariance,
with
\[{\mathcal U}_i(t):= \sqrt{\E \Lambda} \int_0^t {\e}^{\mu_i s}\, \mathrm{d}\tilde W_0(s) +\int_0^t \sqrt{\mu_i\varrho_i(s)}{\e}^{\mu_i s} \, \mathrm{d}\tilde W_i(s),\]
it follows that, for $i\not=k$,
\begin{align*}
{\mathbb C}{\rm ov}(U_i(t),U_k(t))&={\e}^{-(\mu_i+\mu_k) t}\,\E \,
[{\mathcal U}_i(t)\;{\mathcal U}_k(t)]
\\
&={\e}^{-(\mu_i+\mu_k) t}\,\E \Lambda\cdot \E \big[ \int_0^t {\e}^{-\mu_i s}\, \mathrm{d}\tilde W_0(s) \cdot \int_0^t {\e}^{-\mu_k s}\, \mathrm{d}\tilde W_0(s)\big]\\
&={\e}^{-(\mu_i+\mu_k) t}\,\E \Lambda \int_0^t {\e}^{(\mu_i+\mu_k)s}\, \mathrm{d}s
= \E \Lambda /(\mu_i+\mu_k)(1-{\e}^{-(\mu_i+\mu_k)t}).
\end{align*}
This shows (\ref{cov}) for $\alpha>1$.

\vspace{2mm}

\noindent {\it 2.~Slow regime $(\alpha < 1)$.}
In the slow regime, ${\boldsymbol U}^{(N)}(\cdot)$ converges to the solution of 
\[U_i(t) = U_i(0) - \mu_i\,\Psi[U_i](t)+ \big(\sqrt{\Delta\,{\mathbb V}{\rm ar}\,\Lambda} \big)V(t)\, \text{ for } i=1, \dots, d,\] 
which can be written as
\[
\mathrm{d}U_i(t)=- \mu_i U_i(t)\, \mathrm{d}t + \big(\sqrt{\Delta\,{\mathbb V}{\rm ar}\,\Lambda} \big)\,\mathrm{d}V(t) .
\]
Therefore the $U_i(\cdot)$ are Ornstein-Uhlenbeck processes given by:
\begin{align*}
U_i(t)&={\e}^{-\mu_i t} \big(U_i(0) +  \int_0^t \sqrt{\Delta {\mathbb V}{\rm ar}(\Lambda)}\,{\e}^{\mu_i s}\, \mathrm{d}V(s)\big).
\end{align*}
As before, we can conclude that this $d$-dimensional process is Gaussian with expectation vector given by $U_i(0) {\e}^{-\mu_i t}$.
Computations as above reveal that for $\alpha <1$, as claimed  in  (\ref{cov}), 
\[\Cov(U_i(t),U_k(t))=\Delta\Var(\Lambda) / (\mu_i+\mu_k)(1-{\e}^{-(\mu_i+\mu_k)t}).\]

\noindent {\it 3.~Intermediate regime $(\alpha = 1)$}.
In this regime, a combination of the processes from the other cases appears:
\[
\mathrm{d}U_i(t)=-\mu_i U_i(t)\, \mathrm{d}t +\sqrt{\E \Lambda}\,\mathrm{d} \tilde W_0(t) +\sqrt{\mu_i \varrho_i(t)}\, \mathrm{d}\tilde W_i(t)+ \sqrt{ \Delta {\mathbb V}{\rm ar}(\Lambda)}\, \mathrm{d}V(t).
\]
The marginal solutions $U_i(t)$ are, for $i=1, \dots, d$, equal to
\begin{align*}
{\e}^{-\mu_i t}\big( U_i(0)+  \int_0^t \sqrt{\E \Lambda}\,{\e}^{\mu_i s}\, \mathrm{d}\tilde W_0(s)+\int_0^t \sqrt{\mu_i \varrho_i(s)}\,{\e}^{\mu_i s}\, \mathrm{d}\tilde W_i(s)+ \int_0^t \sqrt{{\Delta {\mathbb V}{\rm ar}(\Lambda)}}\,{e}^{\mu_i s}\,\mathrm{d}V(s)\big).
\end{align*}
Again, we conclude that this $d$-dimensional process is Gaussian with expectation vector given by $U_i(0) {e}^{-\mu_i t}$; routine computations yield the desired covariance matrix, as given in (\ref{cov}) and (\ref{covv}).
This completes the proof.
\end{proof}

It is interesting to study the impact of the scaling parameter $\alpha$ on the correlation between the individual queue lengths.
Remarkably,  it turns out that for $\alpha\not=1$ this correlation depends on the service rates only, whereas 
for $\alpha=1$ also the first and second moment of $\Lambda$ play a role; see the following corollary for a result on the stationary regime.

\begin{corollary}[Correlation coefficients]
In stationarity, the correlation coefficient for $i\not=k$ satisfies 
\begin{equation}\label{corrr}
\lim_{N\to\infty} {\mathbb C}{\rm orr}(M_i^{(N)}, M_k^{(N)})=c_{ik}(\alpha) \cdot \frac{\sqrt{\mu_{i} \mu_k}}{\mu_i+\mu_k},
\end{equation}
for some constant $c_{ik}(\alpha) \in [1,2]$. The constant $c_{ik}(\alpha)$ equals $1$ for $\alpha>1$ and
$2$ for $\alpha<1$.
\end{corollary}

\begin{proof}
From Thm.\ \ref{FCLT}, as $t\to\infty$,
\[
C_{ik}(t)\to 
1_{\{\alpha \geqslant1\}}{\displaystyle \frac{\mathbb{E}\Lambda}{\mu_i+\mu_k1_{\{i\not=k\}}} }
+
1_{\{\alpha \leqslant1\}}{\displaystyle \frac{\Delta {\mathbb V}{\rm ar}(\Lambda)}{\mu_i+\mu_k}}
.\]
We observe that (\ref{corrr}) holds, with
\[c_{ik}(\alpha)= \frac{
\E \Lambda\,1_{\{\alpha \geqslant1\}}
+\Delta {\mathbb V}{\rm ar}(\Lambda)\,1_{\{\alpha \leqslant1\}}
}{
\E \Lambda\,1_{\{\alpha \geqslant1\}}
+\frac12 \Delta {\mathbb V}{\rm ar}(\Lambda)\,1_{\{\alpha \leqslant1\}}
}.\]
\end{proof}

\section{Large deviations}\label{sec:LD}
Where the previous section studied the random-environment  infinite-server system under the {\it central limit scaling}, we now focus on the {\it large deviations domain}. 
As it turns out, the previously observed trichotomy remains valid.
The results again translate to the setting with $d$ coupled queues; for ease we first present (and prove) the results for $d=1$, to return to the coupled model at the end of the section.

\subsection{Univariate large deviations}
Let the arrival rate of the $N$-scaled model again be given by $N \Lambda^{(N)}(t)$ (see (\ref{defslab})). 
An important quantity in our analysis is
\[
\kappa_t\big(\Lambda^{(N)}\big)=\int_0^{t}\Lambda^{(N)}(s) \e^{-\mu s}\, \D s =
\frac{1}{\mu}(1-{\e}^{-\mu\Delta N^{-\alpha}})\sum_{j=0}^{\lfloor t/(\Delta N^{-\alpha}) \rfloor - 1}\Lambda_j {\e}^{-\mu j \Delta N^{-\alpha}}  + o_p(1),
\]
as $N \to \infty$.
As observed earlier,  $M^{(N)}(t)$ is a mixed Poisson random variable, with random parameter distributed as $N \kappa_t(\Lambda^{(N)})$.
In the large deviations setting we are interested in the tail probabilities of $M^{(N)}(t)$ for given $t$ and $N$ large.
More specifically, our objective is to evaluate the decay rate
\[\lim_{N\to\infty} N^{-\beta}\log{\mathbb P}\big(M^{(N)}(t)/N \geqslant  a\big),\]
for any $a>\rho(t)= \rho(1-{\e}^{-\mu t})$ (where $\rho:=
\lambda/\mu$) and some specific $\beta>0$.
Given the results obtained in the central limit regime, we expect that $\beta = \min\{1,\alpha\}$.

The main idea behind our analysis is to condition on the value of the random Poisson parameter. In self-evident notation,
\begin{eqnarray}\nonumber
{\mathbb P}\big(M^{(N)}(t)/N \geqslant  a\big)&=&\PR({\rm Pois}(N \kappa_t(\Lambda^{(N)})    \geqslant   N a)\\&=&\int_0^{\infty}  \PR( {\rm Pois}(Nx)    \geqslant   N a) \PR (\kappa_t(\Lambda^{(N)}) \in \dd x).\label{eqP}
\end{eqnarray}
In some parts of the analysis we rely on the following lemma, in which we establish a large deviation result for
$
\PR(\kappa_t(\Lambda^{(N)})   \geqslant   a).
$

\begin{lemma} \label{lemP}
Let $ a > \rho(t)$. Then, with $M(\theta):=\E \,{\e}^{\theta \Lambda}$,
\begin{equation}\label{LDP}
\lim_{N \to \infty} \,{\Delta}{N^{-\alpha}} \log \PR(\kappa_t(\Lambda^{(N)})   \geqslant   a)=- \sup_{\theta > 0} \big(\theta a - \int_0^t \log M(\theta\, {\e}^{-\mu s})\, \dd s\big)
\end{equation}
\end{lemma}
\begin{proof}
As a first step, we define a proxy for $\kappa_t(\Lambda^{(N)})$ that is easier to work with:
\begin{equation}\label{k}
k_t\big(\Lambda^{(N)}\big):= \Delta N^{-\alpha}\sum_{j=0}^{\lfloor t/(\Delta N^{-\alpha})\rfloor-1}\Lambda_j {\e}^{-\mu j \Delta N^{-\alpha}};
\end{equation}
later we show that $\kappa_t(\Lambda^{(N)})$ and $k_t (\Lambda^{(N)})$ are `close enough'.
Let $P_N(a):=\PR(k_t(\Lambda^{(N)})   \geqslant   a)$.
Writing, for arbitrary $\theta>0$,
\[
P_N(a)=\PR\big({\e}^{\theta k_t(\Lambda^{(N)})/(\Delta N^{-\alpha})}    \geqslant   {\e}^{\theta a/(\Delta N^{-\alpha})}\big)
=\PR\big(\prod_{j=0}^{\lfloor t/(\Delta N^{-\alpha}) \rfloor -1} {\e}^{\theta\,\Lambda_j{\e}^{-\mu j\,{\Delta}{N^{-\alpha}}}}   \geqslant   {\e}^{\theta a/(\Delta N^{-\alpha})}\big),
\]
Markov's inequality immediately yields the upper bound
\[
P_N(a) \leqslant {\rm e}^{-\theta a/(\Delta N^{-\alpha})}\prod_{j=0}^{\lfloor t/(\Delta N^{-\alpha}) \rfloor -1} M(\theta {\rm e}^{-\mu j\,{\Delta}{N^{-\alpha}}}).
\]
Recognizing a Riemann sum, we thus obtain
\begin{align*}
\limsup_{N \to \infty} \Delta N^{-\alpha} \log P_N(a) & \leqslant   \limsup_{N \to \infty} \Delta N^{-\alpha}\sum_{j=0}^{\lfloor t/(\Delta N^{-\alpha})\rfloor-1}\log M(\theta {\rm e}^{-\mu j\,{\Delta}{N^{-\alpha}}})-\theta a\\
&=\int_0^t \log M(\theta {\rm e}^{-\mu s}) \, \dd s - \theta a.
\end{align*}
As the established upper bound holds for any $\theta>0$, 
\begin{equation}\label{inf}
\limsup_{N \to \infty} \,{\Delta}{N^{-\alpha}} \log  P_N(a)  \leqslant  \inf_{\theta > 0}\big(\int_0^t \log M(\theta {\rm e}^{-\mu s}) \, \dd s- \theta a\big).
\end{equation}

The next goal is to prove that the above upper bound is tight. 
We do so by first noting that, due to the convexity of the function involved, the infimum in the right-hand side of (\ref{inf}) is attained by  $\theta^\star$, being the unique solution to 
\[
\left.\frac{\partial}{\partial \theta} \int_0^t \log M(\theta {\rm e}^{-\mu s}) \, \dd s \,\right|_{\theta=\theta^\star}= a.
\]
Now we apply a change-of-measure technique.
Define a measure $\Q$ by exponential twisting; the density of the $\Lambda_j$ is changed into
\[
\Q(\Lambda_j \in \dd x) :=  \frac{{\rm e}^{\theta^\star {\rm e}^{- \mu j\,{\Delta}{N^{-\alpha}}}x}}{M(\theta^\star {\rm e}^{-\mu j\,{\Delta}{N^{-\alpha}}})}\PR( \Lambda_j \in \dd x).
\]
Fix an $\varepsilon > 0$, and let the event ${\mathcal E}^{(N)}_{a} :=\{k_t(\Lambda^{(N)}) \in [a,a+\varepsilon)\}$. 
Then
\begin{eqnarray*}
P_N(a)&= &\E_{\Q}\big[
1_{{\mathcal E}^{(N)}_{a}}
\prod_{j=0}^{\lfloor t/(\Delta N^{-\alpha}) \rfloor-1}M(\theta^\star {\e}^{-\mu j\,{\Delta}{N^{-\alpha}}})
\,
{\e}^{-\theta^\star \Lambda_j{\e}^{- \mu j\,{\Delta}{N^{-\alpha}}}}\big]\\
&   \geqslant  & \Q\big(k_t(\Lambda^{(N)}) \in [a,a+\varepsilon)\big) \, {\e}^{-\theta^\star (a+\varepsilon)/(\Delta N^{-\alpha})}\, \prod_{j=0}^{\lfloor t/(\Delta N^{-\alpha}) \rfloor-1}M(\theta^\star {\e}^{-\mu j\,{\Delta}{N^{-\alpha}}}).
\end{eqnarray*}
To obtain that $\Q(k_t(\Lambda^{(N)})\in [a,a+\varepsilon)) \to \frac12$ as $N \to \infty$, we now show that $k_t(\Lambda^{(N)})$ is asymptotically normal.
It is verified  that $\E_{\Q} k_t(\Lambda^{(N)})\to a$ as ${N \to \infty}$, due to the specific construction of the measure $\Q$.
Also,
\begin{eqnarray*}
\lim_{N\to\infty}{\Var_{\Q}(k_t(\Lambda^{(N)}))} \hspace*{-5pt}
&=& \hspace*{-5pt}
\lim_{N\to\infty} \Var_{\Q}\big(\Delta N^{-\alpha}\sum_{j=0}^{\lfloor t/(\Delta N^{-\alpha})\rfloor-1}\Lambda_j \,{\rm e}^{-\mu j \Delta N^{-\alpha}}\big)\\
 \hspace*{-5pt}&=& \hspace*{-5pt}
\lim_{N\to\infty}\Delta N^{-\alpha}\big(\Delta N^{-\alpha} \sum_{j=0}^{\lfloor t/(\Delta N^{-\alpha}) \rfloor-1}  \,{\rm e}^{-2\mu j\Delta N^{-\alpha}}\big){\Var_{\Q}(\Lambda)}\\
\hspace*{-5pt}&=& \hspace*{-5pt}
\lim_{N\to\infty}\Delta N^{-\alpha} \int_0^t {\rm e}^{-2\mu s} {\rm d}s \, {\Var_{\Q}(\Lambda)}=0.
\end{eqnarray*}
Copying the approach -- using cumulant generating functions -- underlying the proof of Theorem \ref{CLT}, it is readily derived that
\[N^{\frac{\alpha}{2}}(k_t(\Lambda^{(N)}) - a) \overset{\mathrm{d}}{\rightarrow} \mathcal{N}(0,\sigma^2),\:\:\:\:
\mbox{where}\:\:\:\sigma^2 := \Delta / (2\mu) (1-{\mathrm e}^{-2\mu t}){\Var_{\Q}(\Lambda)}.\]
Hence, 
\begin{eqnarray*} 
\liminf_{N \to \infty} \,{\Delta}{N^{-\alpha}} \log \PR_N(a) &   \geqslant  & \liminf_{N\to \infty} \,{\Delta}{N^{-\alpha}} \log \Q\big(k_t(\Lambda^{(N)}) \in [a,a+ \varepsilon)\big)\\
&& \qquad -\: \theta^\star (a+\varepsilon)+\int_0^t \log M(\theta^\star {\rm e}^{-\mu s}) \, \dd s  \\
&   \geqslant  & \int_0^t \log M(\theta^\star {\rm e}^{-\mu s}) \, \dd s - \theta^\star (a+\varepsilon).
\end{eqnarray*}
By letting $\varepsilon\downarrow 0$, together with the upper bound this leads to
\begin{equation}
\lim_{N \to \infty} \,{\Delta}{N^{-\alpha}} \log P_N(a)  = -\sup_{\theta > 0}\big(\theta a - \int_0^t \log M(\theta {\rm e}^{-\mu s}) \, \dd s \big).
\end{equation}

Now it remains to show that $k_t(\Lambda^{(N)})$ can again be replaced by $\kappa_t(\Lambda^{(N)})$ (which we abbreviate for compactness to $k_t$ and $\kappa_t$). 
Note that, as $N \to \infty$,
\begin{align*}
|\kappa_t - k_t|=|\big(\frac{1-{\rm e}^{-\mu \Delta N^{-\alpha}}}{\mu \Delta N^{-\alpha}} - 1\big)\Delta N^{-\alpha}\sum_{j=0}^{\lfloor t/(\Delta N^{-\alpha}) \rfloor - 1}\Lambda_j {\e}^{-\mu j \Delta N^{-\alpha}} + o_p(1)|=o_p(1).
\end{align*}
Let $\eta>0$ small enough to guarantee $a - \eta > \rho(t)$.
Then $\PR(\kappa_t \in (k_t - \eta, k_t + \eta) ) \to 1$ as $N \to \infty$, hence
\begin{align}\nonumber
\lim_{N \to \infty} \,{\Delta}{N^{-\alpha}} \log \PR_N(a+\eta) & \leqslant  \lim_{N \to \infty} \,{\Delta}{N^{-\alpha}} \log \PR(\kappa_t \geqslant   a) \leqslant  \lim_{N \to \infty} \,{\Delta}{N^{-\alpha}} \log \PR_N(a- \eta),
\end{align}
which provides bounds for the decay rate of interest of the form 
\begin{equation}\label{rate}
 -\sup_{\theta > 0}\big(\int_0^t \log M {\rm e}^{-\mu s}) \, \dd s- \theta (a \pm \eta)\big).
\end{equation}
The rate function in (\ref{rate}) is continuous in $\eta$, so now letting $\eta\downarrow 0$ yields (\ref{LDP}).
\end{proof}

As in the central limit regime, we distinguish between the cases $\alpha>1$, $\alpha =1$, and $\alpha<1$. 
For all three cases we derive the logarithmic asymptotics.
\vspace{3mm}

\noindent {\it 1. Fast regime} ($\alpha>1$).
We can bound (\ref{eqP}) from below by
\begin{align}\label{under}
&\PR( {\rm Pois}(N (\rho(t)-\varepsilon))    \geqslant   N a) \cdot \PR(\kappa_t(\Lambda^{(N)})\geqslant\rho(t)-\varepsilon),
\end{align}
for some $\varepsilon\in (0, a - \rho(t))$.
As $N$ tends to infinity, it is directly shown that the second factor in (\ref{under}) converges to $1$, and hence has exponential decay rate 0. 
Now an application of Cram\'er's theorem \cite{DZ98} yields
\begin{align*}
\liminf_{N \to \infty} N^{-1} &\log \PR({\rm Pois}(N \, \kappa_t(\Lambda^{(N)}))    \geqslant   N a) \\
&   \geqslant   \lim_{N \to \infty} N^{-1} \log\PR( {\rm Pois}(N (\rho(t)-\varepsilon))    \geqslant   N a)\\
&=-\sup_{\theta} \big( \theta a -  (\rho(t)-\varepsilon)({\rm e}^{\theta}-1)\big)\\
&=a\log\big(\frac{\rho(t)-\varepsilon}{a}\big) -(\rho(t)-\varepsilon)+a.
\end{align*}
On the other hand, (\ref{eqP}) is majorized by
\begin{align}\label{upper}
\PR( {\rm Pois}(N (\rho(t)+\varepsilon))\geqslant N a) + \PR(\kappa_t(\Lambda^{(N)})\geqslant \rho(t)+ \varepsilon). 
\end{align}
By Cram\'er's theorem, the first term in (\ref{upper}) decays exponentially in $N$.
As a consequence of Lemma \ref{lemP}, the second term decays exponentially in  $N^{\alpha}$, i.e., superexponentially in $N$. 
This yields
\begin{align*}
\limsup_{N \to \infty} N^{-1} \log \PR({\rm Pois}(N \kappa_t(\Lambda^{(N)})    \geqslant   N a)& \leqslant  \lim_{N \to \infty} N^{-1} \log\PR( {\rm Pois}(N (\rho(t)+ \varepsilon))    \geqslant   N a)\\
&=a\log\big(\frac{\rho(t)+ \varepsilon}{a}\big) -(\rho(t)+ \varepsilon)+a.
\end{align*}
As this holds for all $\varepsilon > 0$, we conclude that 
\[
\lim_{N\to\infty}N^{-1}\log{\mathbb P}\big(M_N(t)/N \geqslant  a\big)=a\log\big(\frac{\rho(t)}{a}\big) -\rho(t)+a.
\]
Recognizing the decay rate of a Poisson distribution with mean $\rho(t)$, we observe that the essential behavior in the fast regime is again of M/M/$\infty$ type.
\vspace{3mm}

\noindent {\it 2. Slow regime} ($\alpha<1$). 
In this regime we need to distinguish between the situation in which the random variable $\Lambda$ almost surely results in a $\kappa_t(\Lambda^{(N)})$ below $a$, and the situation in which this is not the case.
The proof of the following lemma is straightforward hence omitted.
\begin{lemma}
Given $\Lambda$, let $y= \inf\{x > 0:\PR(\Lambda  \leqslant  x)=1\}$ and $u(t)=y/\mu\,(1-{\sc e}^{-\mu t})$.
Then, as $N \to \infty$,
\[\PR(\kappa_t(\Lambda^{(N)})  \leqslant  u(t)) \to 1.\]
\end{lemma}
The cases $u(t) \geqslant a$ and $u(t) <  a$ should be treated  differently, as follows from the following intuitive explanation that is based on the decomposition (\ref{eqP}). 
If $u(t) \geqslant a$, then the random variable $\Lambda$ can be `large' with respect to $a$, which enables $M^{(N)}(t)$ to reach $Na$ without the Poisson random variable attaining an unlikely value.
If on the contrary $u(t) < a$, then $\Lambda$ is `small' with respect to $a$; $M^{(N)}(t)$ can only exceed level $Na$ by the Poisson random variable attaining an extraordinarily large value.

We first consider the case   $u(t) < a$.
For ease we assume that $\Lambda$ attains values in a discrete set of positive values, of which $y$ is the largest (occurs with probability $p \in (0,1)$) and $y'<y$ the one-but-largest.
It is directly seen that, for $\theta>0$,
\[ \E {\rm e}^{\theta M^{(N)}(t)}  \geqslant  p ^{\lceil t/(\Delta N^{-\alpha}) \rceil}  \,
\E \exp\big(\theta \,{\rm Pois}\big(N u(t)\big)\big).\]
As $\alpha < 1$, this leads to
\begin{equation}\label{slowre} \lim_{N \to \infty} N^{-1} \log \E {\rm e}^{\theta M^{(N)}(t)} \geqslant  u(t)({\rm e}^\theta-1).\end{equation}
In addition, $\E {\rm e}^{\theta M^{(N)}(t)}$ is majorized by 
\begin{align*}
&p\,\E \exp\big(\theta\, {\rm Pois}(Nu(t))\big)+ (1-p)\E \exp\big(\theta\, {\rm Pois}(N (y'/ \mu)(1-\e^{-\mu t}))\big)\\
&=p \exp\big(Nu(t)({\rm e}^{\theta}-1)\big)+ (1-p) \exp\big(
N \,(y'/\mu)(1-\e^{-\mu t})({\rm e}^{\theta}-1)
\big),
\end{align*}
which converges to the right-hand side of (\ref{slowre}) on an exponential scale (use $y>y'$). 
Applying `Cram\'er', we thus find that the probability of interest decays exponentially:
\[
\lim_{N \to \infty}  N^{-1} \log{\PR\big(M^{(N)}(t)/N    \geqslant   a\big)}=-\sup_{\theta>0} \big( \theta a - u(t)({\rm e}^\theta-1)\big)=a\log\big(\frac{u(t)}{a}\big) +a-u(t).
\]

Now we focus on  $u(t) \geqslant a$; 
in this case
\begin{align}\label{lowerr}
\PR( {\rm Pois}(N a)    \geqslant   N a) \PR (\kappa_t(\Lambda^{(N)}) \geqslant a) 
\end{align}
gives an asymptotically non-trivial lower bound for (\ref{eqP}).
Note that for every $\delta > 0$, there is an $N$ large enough such that
\begin{align*}
\PR( {\rm Pois}(N a)    \geqslant   N a) \geqslant   \big(\frac12 - \delta\big),
\end{align*}
so the first factor in (\ref{lowerr}) will not contribute to the decay rate.
The tail behavior of the second factor follows from Lemma \ref{lemP}.
On the other hand, (\ref{eqP}) is majorized by
\begin{align}\label{upperr}
\PR( {\rm Pois}(N (a-\varepsilon))    \geqslant   N a) + \PR(\kappa_t(\Lambda^{(N)})\geqslant a - \varepsilon).
\end{align} 
Again it is observed that only the second term in (\ref{upperr}) contributes to the decay rate: by `Cram\'er' the first term in (\ref{upperr}) decays exponentially, whereas 
the decay of the second term is subexponential (by Lemma \ref{lemP}) for $\varepsilon > 0$ small enough (we need $a- \varepsilon > \rho (t)$).
Letting $\varepsilon \downarrow 0$ while using that the rate function in (\ref{LDP}) is continuous in $a$, we arrive at
\[
\lim_{N \to \infty} \Delta N^{-\alpha} \log{\PR\big(M^{(N)}(t)/N  \geqslant   a\big)}=- \sup_{\theta > 0}\big(\theta a - \int_0^t \log M(\theta {\sc e}^{-\mu s})\, \dd s\big).
\]
Note that the decay rate in this fast regime depends on more detailed information on the distribution of $\Lambda$ than just the mean.
\vspace{3mm}

\noindent {\it 3. Intermediate regime} ($\alpha=1$). 
In this regime we expect exponential decay.
Indeed, it is directly derived that 
\[\lim_{N \to \infty} \Delta N^{-1} \log \E {\rm e}^{\theta M^{(N)}(t)}
= \int_{0}^t \log M(\Delta ({\rm e}^{\theta}-1) {\rm e}^{-\mu s})\, \mathrm{d}s,\]
and hence `G\"{a}rtner-Ellis' \cite{DZ98} gives 
\begin{align*}
\lim_{N \to \infty} \Delta N^{-1} \log{\PR\big(M^{(N)}(t)/N \geqslant   a\big)}&=-\sup_{\theta>0} \big( \theta a - \int_{0}^t \log M(\Delta ({\rm e}^{\theta/ \Delta}-1){\rm e}^{-\mu s})\, \mathrm{d}s
\big).
\end{align*}
For deterministic $\Lambda$ the above result would equal that of the fast regime; the resemblance with the slow regime on the other hand becomes more pronounced for larger values of $\Delta$.

\subsection{Large deviations for the coupled model}
We conclude this section by pointing out how the large devations for the coupled model (where each arrival generates work in $d$ queues) can be dealt with. 
For $\alpha   \geqslant   1$ we are in the regime of exponential decay.  
The multivariate version of the G\"artner-Ellis theorem entails that, modulo the validity of mild regularity conditions to be imposed on the set $A\subset {\mathbb R}_+^d$,
\begin{align*}
&\lim_{N\to\infty} N^{-1} \log {\mathbb P}\big({\boldsymbol M}^{(N)}(t)/N\in A\big)
=-\inf_{{\boldsymbol a}\in A} \sup_{{\boldsymbol\theta}}\big(\sum_{i=1}^d \theta_ia_i-\lim_{N\to\infty}\frac{1}{N}\log {\mathbb E}\exp\big[\sum_{i=1}^d \theta_iM_i^{(N)}(t)\big]\big).
\end{align*}
The problem therefore reduces to characterizing the limiting log moment generating function. It takes standard computations to verify that for $\alpha>1$, with an argumentation borrowed from specific intermediate results in the proof of Lemma \ref{lem:gedrag},
\[\lim_{N\to\infty}N^{-1}\log {\mathbb E}\exp\big[\sum_{i=1}^d \theta_iM_i^{(N)}(t)\big]=
t\,{\mathbb E}\Lambda\big( \int_0^t\frac{1}{t} \prod_{i=1}^d \big(e^{-\mu_i s}(e^{\theta_i}-1) + 1\big){\rm d}s-1\big),\]
whereas for $\alpha=1$ it turns out to equal
\[\frac{1}{\Delta}\int_0^t \log \E \exp\big[\Lambda \Delta\big(\big( \prod_{i=1}^d e^{-\mu_i s}(e^{\theta_i}-1)+1\big)-1\big)\big]{\rm d}s.\]
For $\alpha<1$, as before, the decay is either exponential in $N$ (if the the multi-dimensional random Poisson parameter cannot attain values that are contained in $A$), or exponential in $N^\alpha$. The latter regime being the more complicated one, we here include the corresponding decay rate.  The probability of our interest can be rewritten as 
\begin{equation}\label{eqP2}
\int_{x_1=0}^\infty \cdots \int_{x_d= 0}^\infty F_A({\boldsymbol x}) \cdot \pi(\dd x_1,\ldots, \dd x_d),
\end{equation}
where
\begin{align*}
F_A({\boldsymbol x})&:=\PR\big(\big({\rm Pois}_1(N \, \kappa_t(\Lambda^{(N)})),\ldots,{\rm Pois}_d(N \,\kappa_t(\Lambda^{(N)}))/N \big)  \in A\big),\\
\pi(\dd x_1,\ldots, \dd x_d)&:=\PR (\kappa_{t,1}(\Lambda^{(N)}) \in \dd x_1,\ldots,\kappa_{t,d}(\Lambda^{(N)}) \in \dd x_d);
\end{align*}
here the $d$ Poisson random variables are independent. 
Using the same ideas as above, it can be shown that (\ref{eqP2}) decays exponentially in $N^\alpha$, where the decay rate is now given by
\begin{eqnarray*}
\lim_{N\to\infty} \Delta N^{-\alpha}\log \PR \big(\big(\kappa_{t,1},\ldots,\kappa_{t,d}\big)(\Lambda^{(N)}) \in A\big)
= -\inf_{{\boldsymbol a}\in A} \sup_{{\boldsymbol\theta}}\big( \sum_{i=1}^d\theta_ia_i
-\int_0^t \log M\big(\sum_{i=1}^d\theta_ie^{-\mu_i s}\big){\rm d}s\big).
\end{eqnarray*}

\section{Discussion and future research} 
 
In this paper we propose to model an overdispersed arrival process by a mixed Poisson process in a random environment. 
We assess the impact of overdispersion on system performance when feeding such an arrival process into an infinite-server system. 
Under a specific scaling, we derive (functional) central limit results and large deviations asymptotics.

Various extensions can be explored, a few of which are mentioned here. 
To start with, many results seem to carry over to the setting with generally distributed service times. 
In addition, systematically studying the effect of adding a deterministic trend $\bar{\lambda}(\cdot)$ to the random environment $\Lambda(\cdot)$, the results could be generalized to a setting with nonstationary Cox arrival processes.  
Another challenge lies in refining the logarithmic asymptotics, as obtained in Section \ref{sec:LD}, to exact asymptotics.

In all of the results obtained, we revealed a trichotomy in system behavior depending on the imposed scaling on system size and sampling frequency.
Here the scaling primarily serves to change the level of overdispersion in the system. 
The combination of tunable sampling and tunable overdispersion provides a rich framework for modeling real-world arrival processes.  
One could imagine that in a rapidly changing environment, the inherent overdispersion of the arrival process hardly plays a role, whereas in a slowly changing random environment, overdispersion is expected to be more dominant. 
This interplay between sampling frequency and overdispersion is a convenient feature of our model, which could be used to calibrate the model to real-world data. 
The latter could be a promising direction for future research, involving challenging statistical issues.
 
Another application of our model would be in the area of dimensioning service systems or staffing, and in particular square-root staffing in many-server systems. 
The general idea behind square-root staffing is as follows: a finite-server system is modeled as a system in heavy traffic, where the number of servers $s$ is large and at the same time, the system is critically loaded. 
Under Markovian assumptions this can be achieved by setting $s = \lambda + \beta\sqrt{\lambda}$ (denoting the load on the system by $\lambda$) and letting $\lambda\to\infty$ while keeping $\beta>0$ fixed. 
The system then reaches the desirable Quality-and-Efficiency-Driven (QED) regime, in which the system load approaches 100\% while the delays experienced by customers remain limited. 
In such large-scale service operations, it is natural to use an infinite-server system as a proxy to the many-server system.
Infinite-server models are extremely useful because of their tractability; this can be exploited by translating detailed knowledge of the infinite-server system state to the finite-server setting.
This returns rather good estimates of future arrivals, even in situations of time-varying arrival processes \cite{Whitt99,WGK07}. 
The model developed in this paper provides a new way of modeling such large-scale service systems, with the additional feature of a tunable level of overdispersion, essentially replacing a deterministic $\lambda$ by a stochastically fluctuating $\Lambda$. 
The possibility to design asymptotic dimensioning schemes 
compatible with our new model -- for both static and time-varying overdispersed arrival processes -- is currently investigated by the authors. 

\bibliography{BibliographyArticle1}
\bibliographystyle{plain}
\vfill
\end{document}